\documentclass[12pt]{amsart}
\usepackage{amsmath,amsthm,amssymb,amsfonts}
\usepackage{graphicx}
\usepackage{color}
\usepackage[utf8]{inputenc}

\newtheorem{theorem}{Theorem}[section]
\newtheorem{lemma}[theorem]{Lemma}

\newtheorem{proposition}[theorem]{Proposition}

\newtheorem{corollary}[theorem]{Corollary}
\theoremstyle{definition}
\newtheorem{definition}[theorem]{Definition}

\newtheorem{example}[theorem]{Example}

\newtheorem{remark}[theorem]{Remark}
\numberwithin{equation}{section}

\def\dist{\text{\rm dist\,}}

\def\llll{\longrightarrow}

\newcommand{\C}{{\mathbb {C}}}
\newcommand{\N}{{\mathbb {N}}}

\newcommand{\R}{{\mathbb {R}}}
\newcommand{\T}{{\mathbb {T}}}

\newcommand{\Rea}{\text{\rm Re}\, }

\def\com#1{{``#1''}}

\def\sep{{ \ \  }}
\def\sem{{\ \ \ \  }}
\def\seg{{\ \ \ \  \ \  }}

\title[Stability results] {Stability results of properties related to the  \\ Bishop-Phelps-Bollob{\'a}s  property  for operators 
}

\author[M.D. Acosta]{Mar\'{\i}a D. Acosta}
\address{Universidad de Granada, Facultad de Ciencias,
	Departamento de An\'{a}lisis Matem\'{a}tico, 18071 Granada, Spain}
\email{dacosta@ugr.es}
\author[M. Soleimani]{Maryam Soleimani-Mourchehkhorti}
\address{School of Mathematics, Institute for Research in Fundamental Sciences (IPM), P.O. Box: 19395-5746, Tehran, Iran}
\email{m-soleimani85@ipm.ir}

\thanks{The  first  author was  supported  by Junta de Andaluc\'{\i}a grant  FQM--185  and also by Spanish MINECO/FEDER grant  MTM2015-65020-P. The second author was   supported by a grant from IPM}

\begin{document}
   {\large

\begin{abstract}
	We prove that the class of Banach spaces $Y$ such that the pair $(\ell_1, Y)$ has the Bishop-Phelps-Bollobás property for operators is stable  under finite products when the norm of the product is given by an absolute norm. We also provide examples showing that previous stability results obtained for that property are optimal.
\end{abstract}

\maketitle

   \section{Introduction}
This paper is motivated by recent research on extensions of the so-called Bishop-Phelps-Bollobás  theorem  for operators.  Bishop-Phelps theorem  \cite{BP} states that every continuous linear functional on a Banach space can be approximated (in norm) by norm attaining functionals.   Before to state precisely a
\com{quantitative version} of that result proved by Bollobás \cite{Bol}  we recall some notation.  We denote  by  $B_X$, $S_X$ and $X^*$   the closed unit ball,  the unit sphere
and the topological dual of a Banach space $X$, respectively.  If $X$ and $Y$ are  both real or  both complex Banach spaces, $L(X,Y)$ denotes the space of (bounded linear) operators from $X$ to $Y$, endowed with its usual operator norm.

\vskip3mm

{\it Bishop-Phelps-Bollob{\'a}s Theorem} (see \cite[Theorem 16.1]{BoDu}, or \cite[Corollary 2.4]{CKMMR}). Let $X$ be a Banach space and $0< \varepsilon < 1$. Given $x \in B_X$  and $x^* \in S_{X^*}$ with $\vert 1- x^* (x) \vert < \frac{\varepsilon ^2 }{2}$, there are elements $y \in S_X$ and $y^* \in S_{X^*}$  such that $y^* (y)=1$, $\Vert y-x \Vert < \varepsilon$ and $\Vert y^* - x^* \Vert < \varepsilon $.

A lot of attention has been  devoted to extending Bishop-Phelps theorem  to operators and interesting results have been  obtained about that topic (see for instance \cite{Lin} and  \cite{Bou}). In \cite{Acs}  the reader may find most of the results on the topic known until 2006 and  some open questions on the subject.  The survey paper
\cite{Mars} contains  updated results for  Bishop-Phelps  property  for the space of compact operators.  It deserves to point out that  in general the subset of norm attaining compact operators between two Banach spaces  is not dense in the  corresponding  space of compact operators \cite[Theorem 8]{Mar-j}.

In 2008  the study of extensions of
Bishop-Phelps-Bollob{\'a}s theorem to operators was initiated by Acosta, Aron, Garc{\'i}a and Maestre \cite{AAGM}.  In order to state some of these extensions it will be convenient to  recall  the following notion.

\begin{definition}
	[{\cite[Definition 1.1]{AAGM}}]
	 \label{def-BPBp}
	   Let $X$  and $Y$ be either real or complex Banach spaces. The pair $(X,Y )$ is said to have the Bishop-Phelps-Bollob{\'a}s property for operators (BPBp) if for every $  0 < \varepsilon  < 1 $  there exists $ 0< \eta (\varepsilon) < \varepsilon $ such that for every $T\in S_{L(X,Y)}$, if $x_0 \in S_X$ satisfies $ \Vert T (x_0) \Vert > 1 - \eta (\varepsilon)$, then
	there exist an element $u_0 \in S_X$  and an operator $S \in S_{L(X,Y )}$ satisfying the following conditions
	$$
	\Vert S (u_0) \Vert =1, \sem \Vert u_0- x_0 \Vert < \varepsilon \seg \text{and}
	\sem \Vert S-T \Vert < \varepsilon.
	$$
\end{definition}

In the paper already mentioned it is shown  that the pair $(X,Y)$ has the BPBp whenever $X$ and $Y$ are finite-dimensional spaces \cite[Proposition 2.4]{AAGM}.    The same result also holds true   in case that $Y$ has a certain isometric property
(called property $\beta $ of Lindenstrauss), for  every Banach space $X$
\cite[Theorem 2.2]{AAGM}.  For instance, the spaces $c_0$ and $\ell_\infty$ have such geometric  property.  It is known that  every Banach space  admits an equivalent norm  with the property  $\beta$.  In case that  the domain is $\ell_1$   there is a characterization of the Banach spaces $Y$ such that $(\ell_1,Y)$  has the BPBp \cite[Theorem 4.1]{AAGM}. The geometric property appearing in the previous characterization was called the almost hyperplane series property (in short AHSp) (see Definition \ref{def-AHSP}).

In general  there are  a few  results  about stability of the BPBp under direct sums both on the domain or on the range. For instance, it was shown in \cite[Proposition 2.4]{ACKLM} that  the pairs $\bigl(X, \bigl(\oplus \sum
_{n=1}^\infty Y_n\bigr)_{c_0}\bigr)$ and $\bigl(X, \bigl( \oplus \sum _{n=1}^\infty
Y_n\bigr) _ {\ell_\infty}\bigr)$ satisfy the Bishop-Phelps-Bollob{\'a}s property for
operators whenever all pairs $(X,Y_n)$ have the Bishop-Phelps-Bollob{\'a}s property for
operators \com{uniformly}.  On the other hand,   on the range the BPBp is not stable under
$\ell_p$-sums  for $1 \le p < \infty $ (see \cite[Theorem, p. 149]{Go} and \cite[Theorem 2.3]{Ac1}).
Indeed it is a long-standing open question if for
every  Banach space $X$, the subset of norm attaining operators from $X$ into
the euclidean space $\R^2$ is dense in the corresponding space of operators.

In case that  the domain is $ \ell_1$,   there are some more known results for the stability of the class of Banach spaces $Y$ such that $(\ell_1,Y)$ has the BPBp. In view of the characterization already mentioned, we will list some   known results of stability  of the  AHSp.

As a consequence  of   \cite[Theorem 4.1]{AAGM} and \cite[Proposition 2.4]{ACKLM}, if the family of Banach spaces  $\{ Y_n : n \in \N \}$  has AHSp \com{uniformly}, then  the  spaces $(\bigoplus_{n=1}^\infty Y_n )_{c_0}$ and $(\bigoplus_{n=1}^\infty Y_n )_{\ell_\infty}$  have AHSp. Also it was proved  the stability  of  AHSp under finite $\ell_p$-sums for every $1\leq p < \infty$ \cite[Theorems 2.3 and 2.6]{AAGP}.  Recently this result was extended   to any absolute sum of two summands  (see Definition \ref{def-sum-n}) \cite[Theorem 2.6]{AMS}.  The paper \cite{AMS} also contains some  stability result for $\bigl( \sum _{n=1}^ \infty Y_n  \bigr) _E$, where $E$ is a   Banach sequence space satisfying  certain  additional assumptions  \cite[Theorem 2.10]{AMS}.

The goal of this paper is to obtain some more stability results.
Now we briefly describe the content of the paper. In section $2$  we  recall the  definition of  absolute norm on $\mathbb R ^N$,   the  class of norms induced on a finite  product
of normed spaces by  absolute norms and  some  properties that will be used later.
We also  provide an example showing that, in general, an absolute  norm on $\R^3$ cannot be written   in terms of two absolute norms on $\R^2$ (see Example \ref{ex-ab-norm} for details).

Later in section $3$, we prove that  AHSp is stable  under products  of any finite number of Banach spaces with the same property,  when the product is endowed with  an absolute norm. Notice  that the proof of this  general result  is far from the one for the case of the product of two spaces. We will provide  more detailed arguments in section 3 for that assertion. Let us just mention now  that a simple induction argument does not work in view of Example \ref{ex-ab-norm}.
It is worth to notice that
in general the product of two spaces with AHSp does not necessarily has such property.

In section 4 we show  the parallel stability result for AHp (see Definition \ref{def-AHP}).   Let us mention that
AHp is a property stronger than AHSp.  Finally we provide a simple  example showing that AHSp is not preserved in general by an infinite product  in case that the norm is given by a Banach lattice sequence,
even in the case that all the factors have AHp uniformly. This example shows that the stability result  proved in \cite[Theorem 2.10]{AMS} is  optimal.

\vskip10mm

\section{Definitions and notation}

	In this section we recall the notions of absolute norm on $\R^N$, the norm endowed by an absolute norm on a finite product of normed spaces and  some main properties that we will use later.  We  also recall the notion of approximate hyperplane series property that will be essential in this paper.

	The notion of an absolute norm  for $\C^2$  was introduced  in \cite[\S 21]{BoDu}, where  the reader can find some  properties of these norms.  In different contexts this class of norms has been used   in order to study geometric properties of the direct sum of Banach spaces (see for instance \cite{Pat},  \cite{MPRq} and   \cite{SK}). Although  we will use properties of absolute norms that are  well known we recall the notion that we  use  and state  properties useful to our purpose.
	
	The following notion  is a particular case of the  one  used in  \cite[Section 2]{LMM}. It suffices for our purpose.
	

	\begin{definition}
		A norm $f$ on  $ \R^N$  is called \textit{absolute} if it satisfies that
		$$
		f \bigl( (x_i) \bigr) = f   \bigl( ( \vert x_i \vert )  \bigr) , \sem \forall  (x_i) \in \R^N .
		$$
		An absolute norm $f$ is said to be \textit{normalized}   if $f(e_i)=1$ for every $1\le i \le N $, where $\{ e_i: 1 \le i \le N \}$ is the canonical basis of $\R ^N$. 
	\end{definition}

	Clearly the usual  norms on $\R^N$ are  absolute norms.
	The following statement  gathers some properties of  absolute norms.  Proofs  can be found for instance in \cite[Remark 2.1]{LMM}.  Since  we  consider finite dimensional spaces    next assertions can be also  checked  by using a similar argument to the one used in \cite[Lemmas 21.1 and 21.2]{BoDu}

	\begin{proposition}
		\label{pro-ab-norm}
		Let  $f$  be  an  absolute normalized norm on $\R^ N$. The following assertions hold
		\begin{itemize}
			\item[a)]   If $ x,y \in  \R ^ N $  and $\vert x_i \vert \le \vert y_i \vert$ for each $i \le N $   then  $f ( x )   \le  f ( y ) $.
			\item[b)]    It is satisfied that
			$$
			\Vert x\Vert_\infty  \le  f   (x )  \le
			\Vert x \Vert_1, \sep \forall x  \in \R^N.
			$$
			\item[c)]   If $ x,  y \in \R ^ N $  and $\vert x_i \vert < \vert y_i \vert$ for each $i \le N $   then  $f  (x) < f  (y)  $.
		\end{itemize}
	\end{proposition}
	
	Of course,   the topological dual of $\R^N$ can be identified with   $\R^N$ and the identification is given by the mapping
	$
	\Phi : \R ^N \llll \bigl( \R ^ N \bigr) ^*$ defined by
	$$
	\Phi (y) (x) = \sum _{i =1} ^N  y_i x_i , \sem  \forall y,x \in  \R^N.
	$$
	
	Under this identification, by defining the mapping
	$$
	f^* ( y ) = \max \{ \Phi  ( y )  (x) :  x  \in  \R^N,  f ( x ) \le 1 \},
	$$
	it is immediate that $f^*$ is also an absolute normalized  norm    in case that $f$ is  an absolute normalized  norm on $\R ^N$  and $\Phi$ is a surjective linear  isometry from $(\R^N , f^*)$ to  the dual of the space $(\R^N , f)$.

	 Next concept is standard and has been  used in the literature  very frequently  for the  product  of two spaces
	(see for instance \cite{Behlp},   \cite{MPRis},  \cite{MPRq}, \cite{MPRYjap} and  \cite{HWW})).

	\begin{definition}
		\label{def-sum-n}
	Let   $N$ be  a nonnegative integer,  $X_i$  a Banach space for each $i \le N $ and  $f: \R^N  \llll \R$ be an 	 absolute  norm. Then the mapping  $ \Vert \ \Vert _f : \prod _{ i=1 } ^N  X_i \llll \R$ given by
	$$
	\Vert  x  \Vert _f  = f \bigl( ( \Vert x_i \Vert  ) \bigr), \sem \forall x= (x_i) \in \prod_{ i=1} ^N   X_i
	$$
	is a norm on $ \prod_{ i=1} ^N   X_i $. In what follows, we denote $Z= \prod_{ i=1} ^N   X_i $, endowed with the norm  $\Vert \   \Vert _f $.
	\end{definition}

	The following result  describes the dual and the duality mapping of  the space $Z$, that is essentially well known.  In any case there is a proof in  \cite[Proposition 3.3]{Har}.
	
	\begin{proposition}
		\label{pro-ab-norm-dual}
		Under the previous setting the dual space $Z^*$ can be identified with the space $  \prod_{ i=1} ^N   X_{i}^*  $,  endowed with the absolute norm $f^*$. More precisely, the mapping
		$\psi: \prod_{ i=1} ^N   X_{i}^*  \llll Z^*$  given by
		$$
		\Psi \bigl( (x_{i}^* )\bigr)  \bigl( x_i \bigr) = \sum _{ i =1 }^N  x_{i}^* (x_i), \sem \forall (x_i) \in \prod_{ i=1} ^N   X_i , \  (x_i^*)  \in  \prod_{ i=1} ^N   X_i^*
		$$
		is a surjective linear isometry  from $\prod_{ i=1} ^N  X_{i}^* $  to the topological dual of $Z$,  where we consider in $\prod_{ i=1} ^N   X_{i}^* $ the norm associated to   $f^*$, that is,
		$$
		\Vert  \bigl( x_i^* \bigr)  \Vert _{f^*}  = f^* \bigl( ( \Vert x_i ^*\Vert ) \bigr), \sem \forall (x_i^*)  \in 
		\prod_{ i=1} ^N   X_i^*.
		$$
		Moreover, if $z^*  =\psi \bigl(  (  x_{i}^*) \bigr) \in  S_{ Z^*}$  and $z = (x_i) \in S_Z$, then $z^*(z)=1$ if and only if
		$$
		x_{i}^* (x_i)= \Vert x_{i}^* \Vert \Vert x_i \Vert, \sem \forall i \le N.
		$$
	\end{proposition}

	In  what follows by a convex series we mean a series
	$\sum \alpha_n$ of nonnegative real numbers such that      $\sum _{n=1}^\infty \alpha
	_n =1$.
	Now we recall other   notion essential in our paper which is related to the Bishop-Phelps-Bollobás property for operators.

		\begin{definition}[{\cite[Remark 3.2]{AAGM}}]
		\label{def-AHSP}
			A Banach space $X$  has the {\it approximate
			hyperplane series property} (AHSp) if for every $\varepsilon >0$ there exist
		$\gamma_X\left(\varepsilon\right)>0$ and $\eta_X (\varepsilon )>0$ with
		$\lim_{\varepsilon \to 0}\gamma_X (\varepsilon )=0$ such that for every sequence
		$\{ x_n \}$ in $ S_X$ and every convex series $\sum_{n} \alpha_n$ with
		$$
		\biggl \Vert \sum_{k=1}^{\infty}\alpha_kx_k \biggr \Vert > 1-\eta_X (\varepsilon ),
		$$
		there exist a subset $A\subset \N$ and a subset $\{z_k : k \in A\}\subset S_X$    satisfying the following  conditions 
		\begin{enumerate}
			\item[1)]
			$ \ \sum_{k\in A}\alpha_k>1- \gamma _X(\varepsilon), $
			\item[2)]
			$  \Vert z_k-x_k  \Vert  <\varepsilon \ \ \text{\ for \ all } k\in A$  \ \ {\text and}
			\item[3)]                           
			there is $x^\ast \in S_{X^\ast}$ such that $x^\ast(z_k)=1$ for all  $k\in A.$
		\end{enumerate}
	\end{definition}

	Finite-dimensional spaces, uniformly convex spaces, the classical spaces  $C(K)$  ($ K$ is a compact and Hausdorff space) and $L_1 (\mu)$  ($\mu$ is a positive measure) have AHSp  (see \cite[Section 3]{AAGM}).
	
	 It is convenient to  recall         the following characterization of  AHSp.

	\begin{proposition} [{\cite[Proposition 1.2]{AAGP}}]
		\label{pro-char-AHSP}
		Let $X$ be a Banach space. The following conditions are equivalent.
		\begin{itemize}
			\item[a)] $X$ has the AHSp.
			\item[b)] For every $0 < \varepsilon < 1$ there exist $\gamma_X\left(\varepsilon\right)>0$ and $\eta_X
			(\varepsilon )>0$ with $\lim_{\varepsilon \to 0}\gamma_X (\varepsilon )=0$ such that for every sequence $\{
			x_n \}$ in $ B_X$ and every convex series $\sum_{n} \alpha_n$ with $ \displaystyle{\biggl \Vert
				\sum_{k=1}^{\infty}\alpha_kx_k \biggr \Vert > 1-\eta_X (\varepsilon ),} $ there are a subset $A \subseteq \N
			$ with $\sum_{k\in A}\alpha_k
			>1-\gamma_X (\varepsilon )$, an element $x^* \in  S_{X^*}$, and
			$ \{ z_k: k \in A\} \subseteq \bigl ( x^*\bigr)^{-1} (1) \cap B_X$ such that $\Vert z_k-x_k \Vert
			<\varepsilon$ for all $k\in A.$
			\item[c)] For every $0< \varepsilon < 1$ there exists $0 <
			\eta < \varepsilon$ such that for any sequence $\{ x_n \}$ in $ B_X$ and every convex series $\sum_{n}
			\alpha_n$ with $ \displaystyle{\biggl \Vert \sum_{k=1}^{\infty}\alpha_kx_k \biggr \Vert > 1-\eta,}$ there are
			a subset $A \subset \N $ with $\sum_{k\in A}\alpha_k
			>1-\varepsilon $, an element $x^* \in  S_{X^*}$, and
			$ \{ z_k: k \in A\} \subseteq \bigl ( x^*\bigr)^{-1} (1) \cap B_X$ such that $\Vert z_k-x_k \Vert
			<\varepsilon$ for all $k\in A.$
			\item[d)]  The same statement holds as in $(c)$ but for every
			sequence $\{ x_n\}$ in $S_X$.
		\end{itemize}
	\end{proposition}

	Acosta, Masty{\l}o and Soleimani-Mourchehkhorti proved  that the AHSp is stable under  product of two spaces, endowed with an absolute norm \cite[Theorem 2.6]{AMS}.  The argument for extending that result for more summands is not obvious.   Next we provide       an example of an absolute norm on $\R^3$  that cannot be expreseed in terms of two  absolute norms on $\R^2$.  As a consequence, induction cannot be applied directly  to  prove the stability result of AHSp under  absolute norms.

	\begin{example}
		\label{ex-ab-norm}
		Consider the function on $\R^3$ given by 
		$$
		\vert (x,y,z) \vert = \max \bigl\{ \sqrt{x^2 + y^2 } , \vert x \vert + \vert z \vert \bigr \}
		\sem  \bigl(  (x,y,z ) \in \R ^3 \bigr).
		$$	
		Then $\vert \  \; \vert $ is an absolute normalized norm on $\R^3$ and there are no absolute norms $f$ and  $g$ on $\R^2$ satisfying any of the following three assertions
		\begin{enumerate}
			\item[\textbf{i)}]	 $\vert (x,y,z) \vert = f ( g(y,z),x), \sep \forall  (x,y,z) \in \R^3.$
			\item[\textbf{ii)}]	  $\vert (x,y,z) \vert = f ( g(x,z),y), \sep \forall  (x,y,z) \in \R^3.$
			\item[\textbf{iii)}]	  $\vert (x,y,z) \vert = f ( g(x,y),z), \sep \forall  (x,y,z) \in \R^3.$	
		\end{enumerate}
	\end{example}
	\begin{proof}
		It is immediate to check that $\vert \ \; \vert$ is an  absolute normalized norm on $\R^3$.
		
		\noindent
		\textbf{i)} Assume that  it is satisfied the equality
		$$
		\vert (x,y,z) \vert = f ( g(y,z),x), \sem \forall  (x,y,z) \in \R^3.
		$$
		Since $\vert e_2 \vert = \vert e_3 \vert =1$ we have that
		$$
		1= f(g(1,0),0) = g(1,0) f(1,0) , \sem  1= f(g(0,1),0) = g(0,1) f(1,0)
		$$
		and so
		$$
		g(1,0)= g(0,1).
		$$
		As a consequence we obtain that
		$$
		\sqrt{2}= \vert ( 1,1,0) \vert = f(g(1,0),1) =  f(g(0,1),1) =  \vert ( 1,0,1) \vert =2,
		$$
		which is a contradiction. So condition i) cannot be satisfied.
		
		\noindent
		\textbf{ii)} Assume now that   it is satisfied
		$$
		\vert (x,y,z) \vert = f ( g(x,z),y), \sem \forall  (x,y,z) \in \R^3.
		$$
		So
		\begin{equation}
		\label{g-ii}
		\vert x \vert + \vert z \vert  = \vert (x,0,z)\vert  = f(g(x,z),0) = g(x,z) f(1,0), \sem
		\forall (x,z) \in \R ^2.
		\end{equation}
		 Hence we obtain that
		$$
		\sqrt{x^2 + y^2}= \vert (x,y,0) \vert =  f (  g(x,0),y)= f\Bigl( \frac{x}{ f(1,0)}, y\Bigr), \sem
		\forall (x,y) \in \R ^2 .
		$$
		That is,
		$$
		f(x,y) = \sqrt{ (f(1,0) x ) ^2 + y^2 }, \sem \forall (x,y)\in \R^2.
		$$
		As a consequence, in view of the previous equality and \eqref{g-ii} we deduce that
		$$
		f(g(x,z),y) ) =   \sqrt{ (f(1,0) g(x,z) ) ^2 + y^2 } =  \sqrt{ (\vert x \vert + \vert z \vert ) ^2 + y^2 }, \sem \forall (x,y,z)\in \R^3.
		$$
		But the last equality contradicts the assumption of ii).

		\noindent
		\textbf{iii)} Assume now that
		$$
		\vert (x,y,z) \vert = f(g(x,y),z), \sem \forall (x,y,z) \in \R^3.
		$$
		Hence we get that
		\begin{equation}
		\label{g-x-y-sqrt}
		\sqrt{x^2 + y^2 }  = f(g(x,y),0) = g(x,y) f(1,0), \sem \forall (x,y) \in \R^2.
			\end{equation}
			As a consequence  we have that
		$$
		\vert x \vert +  \vert z \vert = \vert (x,0,z) \vert = f (g(x,0),z )=  f  \Bigl(  \frac{ x  } { f(1,0)}, z \Bigr) , \sem  \forall (x,z) \in \R^2,
		$$
		that is,
		\begin{equation}
		\label{f-iii}
		f(x,z)= f(1,0) \vert x \vert + \vert z \vert, \sem \forall (x,z) \in \R^2.
		\end{equation}
		 For each $(x,y,z) \in \R^3,$  in  view of \eqref{f-iii} and \eqref{g-x-y-sqrt}   we obtain that
		\begin{align*}
		\max \bigl\{ \sqrt{ x^2 + y^2}, \vert x \vert + \vert z \vert \bigr\} & =  f(g(x,y),z) \\
		& =  f(1,0) g(x,y) + \vert z \vert \sem \forall (x,y,z)\in \R^3 \\
		& =  \sqrt{ x^2 + y^2} + \vert z \vert,
		\end{align*}
		which is a contradiction. So $\vert \  \; \vert$ cannot satisfy condition iii).
	\end{proof}

\section{ Stability result of the approximate hyperplane series property}

As we already mentioned in the introduction, the goal  of this section is to prove
that  the AHSp is stable under finite products in case that the norm  of the product
is given by an absolute norm. For product of two spaces that result was proved in \cite[Theorem 2.6]{AMS}.

In the proof of the stability of AHSp for the product of two spaces Lemma 2.5 in  \cite{AMS} plays an essential role. But the statement of that result does not hold in case that  we replace $\R$ by  $\R^2$. For instance,  this is the case of the absolute norm on $\R^3$ whose closed unit ball  is the convex hull of the set given by
$$
\{ (x,y,0): x^2 + y^2 \le 1\} \cup
\{ (x,0,z): x^2 + z^2 \le 1\}
$$
$$
 \cup
\{ (0,y,z): y^2 + z^2 \le 1\} \cup
\Bigl\{ \frac{1}{ \sqrt{2}} (r,s,t): r,s,t \in \{1,-1\} \Bigr\}.
$$

 The following  result is a consequence of \cite[Lemma 3.3]{AAGM}.

\begin{lemma}
	\label{elemental}
	Let   $\{z_k\}$ be a sequence  of complex numbers with $\vert z_k \vert \le 1$  for any nonnegative integer  $k,$ and
	let $  0 < \eta < 1$ and   $\sum  \alpha _k$ be   a convex series   such that  \ $ \Rea \sum_{ k=1}^{\infty}
	\alpha _k z_k > 1 - \eta^2   $. If we define $A:= \{
	k \in {\mathbb{N}}: \Rea z_k > 1-\eta \}$ then 
	$$
	\sum _{ k \in A} \alpha _k   >  1 - \eta.
	$$
\end{lemma}

The next statement  is a refinement of \cite[Lemma 3.4]{AAGM}  that will be very useful.

\begin{lemma}
	\label{le-AHSP-strong}
	Assume that $\vert \  \; \vert$ is a norm on $\R^N$. Then for every $\varepsilon > 0$, there is $ \delta > 0$ such that
	whenever $ a^\ast \in S_{(\R^N) ^\ast}$,  there exists $\ b^\ast \in S_{(\R^N) ^\ast} $ satisfying  $ 	\text{\rm dist} (a, F( b^\ast)
	)< \varepsilon$ for all $ a \in \{ z \in S_{\R^ N}:  \  a^\ast (z) > 1 -
	\delta \} $,  where $F(b^\ast):= \{ y\in S_{\R^ N}: b^\ast (y)=1\} $ and also $b^* (e_i) =0$ for every $i \le N$ such that $a^* (e_i) =0$.
\end{lemma}
\begin{proof}
	For a subset $G\subset \{ k \in \N: k \le N\}$ we define
	$$
	Z_G := \{ z^* \in  S_{(\R^N) ^\ast} : z^* (e_i)=0, \forall i \in G \}.
	$$
	It is clear that $Z_G$ is a compact set of $(\R^N)^*$.
	
	We argue by contradiction.  So assume that there is a set  $G\subset \{ k \in \N: k \le N\}$,  some positive real number 	$\varepsilon_0$   such that  for each  $\delta > 0$ there is  $a^{*}_\delta   \in Z_G$  such that for each  $b^*  \in Z_G$ there is  some element $a \in \{ z \in S_{\R ^N}:  a^{*}_\delta  (z) > 1 - \delta \}$ such that $\dist (a, F( b^{*})) \ge \varepsilon_0 $.
	
	So there are  sequences  $(r_n) \to 1$,  $ (a_{n}^\ast) \subset
	Z_G $  such that for all $b^\ast \in Z_G, \{ a \in S_{\R ^N} : a_{n}^\ast (a)
	> r_n\} \cap \{a \in  S_{\R ^N}:  \dist ( a, F(b^\ast)) \ge \varepsilon _0 \} \ne \varnothing. \
	$
	By compactness of  $Z_G$, we may assume  that $( a_{n }^\ast ) \to a^\ast$ for
	some $a^\ast \in Z_G$.  By the previous condition there is a sequence $(a_n) $ in $   S_{\R ^N}$  satyisfying  $ r_n <   a_{n}^\ast (a_n) \le 1 $ for each  $n$ and such 		that 
	\begin{equation}
	\label{dist-grande}
	\text{\rm dist} (a_n, F(a^\ast)) \ge \varepsilon_0 , \sem \forall n \in \N
	\ .
	\end{equation}
	By passing to a subsequence, if needed, we also may  assume that $(a_n)$ converges to some $a \in  S_{\R ^N}$.  Since $ ( a_{n}^\ast (a_n)
	) \to 1$ and both sequences are convergent, it follows that $a^\ast (a) =1$; that is, $a \in F(a^\ast ) $.  As a consequence we obtain that $\dist (a_n , F(a^\ast)) \le \Vert
	a_n - a \Vert $ for every $n$. Since $( a_n) $ converges to $ a$, the previous inequality contradicts \eqref{dist-grande}.
\end{proof}

\begin{theorem}
	\label{th-stable}
	Assume that  $\vert \  \  \vert$ is an absolute    normalized norm on $\R^N$ and $\{ X_i: i \le N\}$ are Banach spaces having the AHSp, then $Z= \prod_{i=1}^N X_i$ has the AHSp, where $Z$ is endowed with the norm given by
	$$
	\Vert (x_1, \ldots, x_N ) \Vert = \vert ( \Vert x_1\Vert, \ldots, \Vert x_N \Vert ) \vert, \seg (x_i \in X_i, \forall i \le N).
	$$
\end{theorem}
\begin{proof}
	For a set $G \subset \{ k \in \N : k \le N\}$  we define $P_G: Z \llll Z$ by 
	$$
	P_G (z)(i)= z_i \sep \text{if} \sep i \in G \sem \text{and}  \sem  P_G (z)(i)= 0 \sep \text{if} \sep i \in \{1,2, \ldots, N\}  \backslash G.
	$$
	For each $i \le N$ we denote by $Q_i(z)= z_i$ for every $z \in Z$. 
	
	We can clearly assume that $X_i \ne \{0\}$ for each $i \le N$. We  will prove the result by induction on $N$. For $N=1$ the result is trivially satisfied. So we assume that $ N \ge 2$ and  the result is true for the space  $\prod _{i\in G} X_i$ for any subset $G \subset 
	\{ k \in \N : k \le N\}$  such that $\vert G \vert \le N-1$.  We will prove the result for $G= \{ k \in \N : k \le N\}$.   To this end we use that in view of \cite[Proposition 3.5]{AAGM} finite-dimensional spaces have AHSp. 
	
	Assume that $ 0 < \varepsilon < 1$ and let $\eta :]0,1[ \llll  ]0,1[ $ be a function such that 
	\newline
	{\bf a)} the pair $(\varepsilon, \eta(\varepsilon) )$ satisfies  condition c) in Proposition \ref{pro-char-AHSP} for the space  $(\R^N, \vert \  \vert )$ and for the Banach spaces $\prod _{i \in G} X_i$ for each $G \subset \{ k \in \N : k \le N\}$  such that $\vert G \vert \le N-1$,
	\newline
	{\bf b)} the pair $(\varepsilon, \eta(\varepsilon) )$ satisfies Lemma \ref{le-AHSP-strong} for $\delta = \eta (\varepsilon)$ and 
	\newline
	{\bf c)} $\eta (\varepsilon) < \varepsilon$ for every $ \varepsilon \in ]0,1[$.

	We  will show that $Z$ satisfies condition d) in Proposition \ref{pro-char-AHSP} for $ \eta ^\prime  = \biggl( \dfrac{ \eta \bigl( \eta \bigl( \frac{\varepsilon}{4N} \bigr)\bigr)}{ 2N}\biggr) ^8$. Assume that $(u_k)$ is a sequence in $S_Z$ and $\sum \alpha _k$  is a convex series such that
	$$
	\biggl \Vert \sum _{k=1}^\infty \alpha _k u_k \biggr \Vert > 1 - \eta ^\prime .
	$$
	By Hahn-Banach theorem there is a functional $u^* =(u_{1}^*,u_{2}^*, \ldots, u_{N}^*  )\in S_{Z^*}$ such that 
	\begin{equation}
	\label{norma-big}
	1 - \eta ^\prime  < \Rea u^* \biggl( \sum _{k=1}^\infty \alpha _k u_k \biggr ) =
	\Rea  \biggl( \sum_{k=1}^\infty  \alpha _k \Bigl(  \sum _{i=1}^N   u_{i}^* (u_k(i))  \Bigr)\biggr ). 
	\end{equation}

	Now we define the set   $ F \subset \{ k \in \N : k \le N\}$  by 
	$$
	F= \Bigl\{ i\le N : \Vert u_{i}^* \Vert > \sqrt[8]{\eta ^\prime } \Bigr \} \sem \text{and} \sem F^c = \{ i \in \N: i \le N \} \backslash  F.
	$$
	Since $u^* \in  S_{Z^*}$, in view of Proposition \ref{pro-ab-norm-dual}  and assertion b) in Proposition \ref{pro-ab-norm} we obtain that  $F \ne \varnothing$.   We consider two cases.
	
	{\bf Case 1.} Assume that $\vert F \vert < N$.
	\newline
	Notice that
	\begin{align*}
	1 - \eta ^\prime	& <  \Rea \sum _{k=1}^\infty \alpha _k  u^* \bigl( u_k \bigr ) \\
	& = \Rea \sum _{k=1}^\infty \alpha _k   \Bigl( \sum _{i \in F} u_{i}^* (u_k(i)) \Bigr)  +  \Rea \sum _{k=1}^\infty \alpha _k   \biggl( \sum _{  i \in F^c} u_{i}^* \bigl(u_k(i) \bigr) \biggr)	\\
	& \le  \Rea \sum _{k=1}^\infty \alpha _k   \Bigl( \sum _{i \in F} u_{i}^* (u_k(i) ) \Bigr)  +  \sum _{k=1}^\infty \alpha _k  \biggl( \sum _{  i \in F^c}
	\sqrt[8]{\eta ^\prime } \biggr) \\
	& \le  \Rea \sum _{k=1}^\infty \alpha _k   \Bigl( \sum _{i \in F} u_{i}^* (u_k(i)) ) \Bigr)  +   N \sqrt[8]{\eta ^\prime} \\
	& \le  \Rea \sum _{k=1}^\infty \alpha _k   \Bigl( \sum _{i \in F} u_{i}^* (u_k(i) ) \Bigr)  +   \frac{ \eta \bigl( \eta \bigl( \frac{\varepsilon}{4N}\bigr)\bigr)}{2}. 
	\end{align*}
	So  
	\begin{equation}
	\label{re-F-big}
	\Rea \sum _{k=1}^\infty \alpha _k   \Bigl( \sum _{i \in F} u_{i}^* (u_k(i) ) \Bigr) > 1- \eta ^\prime -   \frac{ \eta \bigl( \eta \bigl( \frac{\varepsilon}{4N}\bigr)\bigr)}{2} > 1 - 
	\eta \Bigl( \eta \Bigl( \frac{\varepsilon}{4N}\Bigr)\Bigr). 
	\end{equation}
	
	By assumption the space $\prod _{i \in F} X_i$ has AHSp,  and in view of a)   there is a set $A \subset \N$ and $v^* = \bigl( v_{i}^*  \bigr) _{i \in F}  \in S_{ ( \prod _{i\in F} X_i )^*}$ such that
	\begin{equation}
	\label{sum-alpha-A}
	\sum_{k \in A} \alpha _k > 1 - \eta \Bigl( \frac{\varepsilon}{4N} \Bigr) > 1- \frac{\varepsilon}{4N} > 1 -\varepsilon
	\end{equation}
	and for every $k \in A$ there is $v_k \in  S_{\prod _{i\in F} X_i }$  such that
	\begin{equation}
	\label{v*-vk}
	v^* (v_k) = \sum _{i \in F} v_{i}^* \bigl( v_k(i) \bigr) = 1, \sem \forall k \in A
	\end{equation}
	and 
	\begin{equation}
	\label{vk-close-QFuk}
	\bigl \Vert v_k - P_F (u_k) \bigr \Vert < \eta \Bigl( \frac{\varepsilon}{4N} \Bigr) < \frac{\varepsilon}{4N} < \frac{\varepsilon}{4}, \sem \forall k \in A.
	\end{equation}
	Now we define $G $ as follows
	$$
	G = \bigl\{ i \in F : \exists k \in A, v_{i}^* (v_k (i)) \ne 0\bigr\}.
	$$
	By \eqref{v*-vk} we have that
	\begin{equation}
	\label{sum-G-y_k}
	\sum _{i \in G } v_{i}^* \bigl( v_k (i))\bigr) = \sum _{i \in F } v_{i}^* \bigl(  v_k(i)\bigr) =  v^* (v_k)=1, \seg \forall k \in A.
	\end{equation}
	As a consequence  $(v_{i}^* ) _{ i \in G } \in S_{ ( \prod _{ i \in G} X_i ) ^*} $.
	In view of  Proposition \ref{pro-ab-norm-dual} we  have that 
	\begin{equation}
	\label{vi*-vk}
	v_{i}^* \bigl( v_k(i)\bigr) = \Vert v_{i}^*  \Vert \Vert v_k (i)  \Vert , \seg \forall i \in F, k \in A.
	\end{equation}
	
	Now we define the element $w^* \in Z^*$ as follows
	$$
	w_{i}^* = \begin{cases}
	v_{i}^* \sep & \text{if} \sep i \in G \\
	0 \sep & \text{if} \sep i \in  \{ j\in \N: j \le N \} \backslash G.
	\end{cases}
	$$ 
	It is trivially satisfied that $w^* \in S_{Z^*}$ and by \eqref{sum-G-y_k}  we have
	\begin{equation}
	\label{w*-yk}
	w^* (v_k) = \sum _{ i \in G } w_{i}^*  \bigl( v_k(i)\bigr) = 
	\sum _{ i \in G } 
	v_{i}^* \bigl( v_k(i)) \bigr)   =1 , \sem \forall  k \in A.
	\end{equation}
	So by \eqref{vk-close-QFuk} for each $k \in A$  we have that
	$$
	\Rea w^* (u_k) =  \Rea w^* \bigl( P_F(u_k) \bigr)  \ge \Rea w^* (v_k) - \bigl \Vert  v_k - P_F (u_k) \bigr \Vert > 1 - \eta \Bigl( \frac{\varepsilon}{4N} \Bigr).
	$$
	That is, for each  $k \in A$ it is  satisfied that
	$$
	1 - \eta \Bigl( \frac{\varepsilon}{4N} \Bigr) < \Rea \sum_{i=1}^N w_{i}^* \bigl( u_k(i)) \bigr) \le   \sum_{i=1}^N \Vert  w_{i}^* \Vert  \Vert  u_k(i) \Vert =  \sum_{i \in G } \Vert  v_{i}^* \Vert  \Vert  u_k(i) \Vert. 
	$$
	By using  condition b)    there exists $s= (s_1, s_2, \ldots, s_N) \in S_{ (\R^ N)^*}$ such that for every $i \in \{ k \in \N: k \le N\} \backslash G$, 
	$s_i=0$ and for every $k \in A$ there exists $r_k = (r_k(i))_{ i \le N} \in S_{\R^N }$ such that
	\begin{equation}
	\label{s-r-1}
	\sum _{k=1}^N s_i r_{k} (i)=1, \sem \Bigl \vert  \bigl( r_{k}(i) \bigr) _{i\le N } - 
	\bigl( \bigl \Vert u_k(i)) \bigr \Vert \bigr) _{i\le N } \Bigr \vert < \frac{\varepsilon}{4N }  < \frac{\varepsilon}{4 }  .
	\end{equation}
	Finally we define $z^* = \bigl( z_{i}^*\bigr)_{i\le N} \in Z^*$ as follows
	$$
	z_{i}^* = \begin{cases}
	s_i \frac{v_{i}^*}{\Vert v_{i}^* \Vert }  \sep & \text{if} \sep i \in G\\
	0 \sep & \text{if} \sep i \in  \{ j \in \N: j \le N \} \backslash G.
	\end{cases}
	$$ 
	By Proposition \ref{pro-ab-norm-dual} we have that $\Vert z^*\Vert = \Vert  \bigl( s_i\bigr)  _{i \le N }  \Vert _{ (\R ^N ) ^*} = 1 $, so $z^* \in S_{Z^*}$.
	
	Notice that for every $i \in G$ there exists $k_0^i \in A$ such that $v_{i} ^* ( v_{k_0^i} (i))\ne 0$. 
	For every $i \in   \{ j \in \N: j \le N \} \backslash G $ we choose $x_{i} \in S_{ X_i}$ and for every $k \in A$ we define $z_{k} \in S_Z$ as follows 
	$$
	z_{k}(i) = \begin{cases}
	r_{k}(i)  \frac{v_k(i) }{\Vert v_k(i)  \Vert }  \sep & \text{if} \sep i \in F \sep \text{and} \sep v_k(i)\ne 0 \\
	r_{k}(i)  \frac{v_{k_0^i}(i) }{\Vert v_{k_0^i}(i)  \Vert }  \sep & \text{if} \sep i \in G  \sep \text{and} \sep v_{k}(i)=0 \\
	r_{k}(i) x_i  \sep & \text{if} \sep i \in F \backslash  G  \sep \text{and} \sep v_k(i)=0 \\
	r_{k}(i)  \frac{ u_k (i) }{\Vert u_k (i)  \Vert }  \sep & \text{if} \sep i \in  \{ j \in \N : j \le N \}  \backslash F  \sep \text{and} \sep u_k (i)\ne 0 \\
	r_{k}(i) x_i  \sep & \text{if} \sep i \in
	  \{ j \in \N : j \le N \}   \backslash F  \sep \text{and} \sep u_k (i)=0.
	\end{cases}
	$$ 
	Since  $\Vert z_{k}\Vert = \bigr\vert \bigl( r_{k} (i) \bigr) _{i\le N} \bigr \vert = 1$ we have that  $z_k \in S_Z$ for every $k \in A$.  By  \eqref{vi*-vk} and \eqref{s-r-1}, taking into account that $s_i=0$ for each $i \in \{ j \le N\} \backslash G$,  it is also satisfied that
	\begin{equation}
	\label{z*-zk}
	z^* (z_k) = \sum_{i \in G } s_i r_{k}(i)=  \sum_{i=1}^N s_i r_{k}(i)=1, \seg \forall k \in A . 
	\end{equation}
	
	Let us fix $k \in A$. 	For $i \in F$ it is clear that 
	\begin{eqnarray}
	\label{zki-uki}
	\bigl\Vert  z_{k} (i) - u_k (i) \bigr\Vert & \le&  \bigl\Vert  z_{k} (i) - v_k (i) \bigr\Vert + \bigl\Vert   v_k (i)- u_k (i)  \bigr\Vert  \notag \\
	&=&  \bigl\vert  r_{k} (i) - \bigl \Vert v_k (i)  \bigr\Vert \;  \bigr \vert + 
	\bigl\Vert   v_k (i)- u_k (i) \bigr\Vert \notag\\
	&\leq&    \bigl\vert  r_{k} (i) - \bigl \Vert u_k (i)  \bigr\Vert \;  \bigr \vert +  2 	\bigl\Vert    v_k (i) - u_k (i)  \bigr\Vert \notag.
	\end{eqnarray}
	As a consequence, by using also \eqref{s-r-1} and \eqref{vk-close-QFuk} we obtain that 
	
	\begin{align}
	\label{QF-zk-QF-uk}
	\nonumber
	\bigl \Vert P_F (z_k) - P_F (u_k) \bigr \Vert 		&  \le   \bigl \vert \bigl( r_{k}(i) \bigr)_{i\le N }  - \bigl ( \bigl \Vert u_k (i) \Vert \bigr) _{i \le N } \bigr \vert + 2 \bigl \Vert P_F (v_k) -  P_F (u_k) \bigr \Vert\\
	&  <  \frac{\varepsilon}{4} + \frac{2\varepsilon}{4} = \frac{3\varepsilon}{4} .
	\end{align}
	
		For $i \in \{j \in \N: j \le N\} \backslash F$ we have that
	$
	\Vert  z_{k} (i)  - u_k (i) \Vert = \bigl\vert  r_{k}(i) - \Vert u_k (i) \Vert \; \bigr \vert $, so in view of \eqref{s-r-1} we obtain that
	\begin{equation}
	\label{QFc-zk-QFc-uk}
	\bigl \Vert P_{F^c} (z_k) - P_{F^c} (u_k) \bigr \Vert 	 \le  \bigl \vert \bigl( r_{k}(i) \bigr)_{i\le N }  - \bigl ( \bigl \Vert u_k (i) \Vert \bigr) _{i \le N }  \bigr \vert<  \frac{\varepsilon}{4} .
	\end{equation}
	From \eqref{QF-zk-QF-uk} and \eqref{QFc-zk-QFc-uk} we conclude that $\Vert z_k - u_k \Vert < \varepsilon $ for every $k \in A$. Since we know that $z^* \in S_{Z^*}$ and by  \eqref{z*-zk} and  \eqref{sum-alpha-A}  the proof is  finished  in case  1.

	{\bf Case 2.} Assume now that $F=\{ i \in  \N : i \le N\}$. We define the set $B$ by
	$$
	B=\Bigl\{ k \in \N : \Rea  u^* (u_k) > 1 -\sqrt{\eta ^\prime} \Bigr\}.
	$$
	In view of \eqref{norma-big} and Lemma \ref{elemental} we obtain  that 
	\begin{equation}
	\label{sum-alpha-B}
	\sum _{k \in B} \alpha _k > 1 - \sqrt{\eta ^\prime}.
	\end{equation}
	In view of Proposition \ref{pro-ab-norm-dual}, for every $k \in B$ we have that
	\begin{align}
	\label{u*-uk}
	\nonumber
	1 - \eta \Bigl( \eta \Bigl( \frac{ \varepsilon}{4N} \Bigr) \Bigr)& < 1 - \sqrt{\eta ^\prime }  \\
	\nonumber
	& < \Rea u^* (u_k)  \\
	& = \Rea \Bigl( \sum _{i=1}^N u_{i}^* \bigl( u_k(i)) \Bigr) \\
	\nonumber
	& \le   \sum _{i=1}^N \bigl \Vert  u_{i}^* \bigr \Vert \;  \bigl \Vert  u_k(i) \bigr \Vert \le 1.
	\end{align}
	By condition b) there is $s=(s_1, s_2, \ldots, s_N ) \in S_{(\R^N) ^*}$ and for every $k \in B$ there is $\bigl( r_{k} (i) \bigr) _{i\le N} \in  S_{\R^N}$  such that
	\begin{equation}
	\label{s-r-1-second}
	\sum _{k=1}^N s_i r_{k} (i)=1, \sem \Bigl \vert  \bigl( r_{k}(i) \bigr) _{i\le N} - 
	\bigl( \bigl \Vert u_k(i)) \bigr \Vert \bigr) _{i\le N} \Bigr \vert < 
	\eta \Bigl( \frac{\varepsilon}{4N} \Bigr) <  \frac{\varepsilon}{4N}, 
	\end{equation}
	where we also used that $\eta $ satisfies condition c). 
	From \eqref{u*-uk} for each $k \in B$ we have
	\begin{equation}
	\label{ui*-uki}
	\Rea u_{i}^* \bigl( u_k(i)) \bigr) \ge \Vert u_{i}^* \Vert \;   \Vert u_k(i) \Vert - \sqrt{\eta ^\prime }, \seg \forall 1 \le i \le N.
	\end{equation}
	
	Now for each $i \le N$ we define the  set $C_i \subset B$ as  follows
	$$
	C_i=\Bigl\{ k \in B:   \bigl\Vert u_k(i)  \bigr \Vert > \sqrt[8]{\eta ^\prime} \Bigr\}.
	$$
	Since  $F =\{ i\in \N : i \le N\}$, for each $i \le N$ we know that $\Vert u_{i}^* \Vert > \sqrt[8]{\eta ^\prime}$.  Hence  for each $i \le N$  
	such that $ C_i \ne \varnothing $ 
	  from  \eqref{ui*-uki} we obtain that 
	$$
	\frac{\Rea u_{i}^* \bigl( u_k(i)) \bigr) }{\bigl \Vert u_{i}^* \bigr \Vert \; \bigl \Vert  u_k(i) \bigr \Vert }  
	\ge  1 - \frac{\sqrt{\eta ^\prime  }}{\sqrt[8]{\eta ^\prime  } \sqrt[8] {\eta ^\prime }}= 1- \sqrt[4]{\eta ^\prime  }  > 1- \eta \Bigl( \eta \Bigl( \frac{\varepsilon}{4N} \Bigr) \Bigr), \sem  \forall k \in C_i.
	$$
	Since $X_i$ has AHSp, by using a)  there is a set $D_i \subset C_i$ such that
	\begin{equation}
	\label{sum-alpha-D_i}
	\sum_{ k \in D_i } \alpha _k \ge \Bigl( 1 - \eta \Bigl( \frac{\varepsilon}{4N} \Bigr) \Bigr) \sum_{ k \in C_i } \alpha _k
	\end{equation}
	and there is  $v_{i}^* \in S_{ (X_i)^*} $ and for every $k \in D_i$ there is  $v_{k}(i) \in S_{ X_i}$ such that
	\begin{equation}
	\label{vi*-vki-second}
	v_{i}^* \bigl( v_k(i) \bigr) =1, \seg \biggl \Vert v_{k}(i) -  \frac{u_k(i) }{\bigl \Vert  u_k(i) \bigr \Vert }  \biggr \Vert  <  \eta \Bigl( \frac{\varepsilon}{4N} \Bigr).
	\end{equation}
	 In case that  $C_i =\varnothing $  for some $i \le  N$  we take $D_i= \varnothing$.	Now we define the set $E \subset B$  by $ E= \bigcap_{i=1}^N  \bigl( D_i \cup ( B \backslash C_i ) \bigr) $.
	Notice that for every $1 \le i \le N$ we have 
	\begin{align*}
	\sum _{ k \in D_i \cup (B\backslash C_i)} \alpha _k 	& = \sum _{ k \in D_i } \alpha _k  + \sum _{ k \in B\backslash C_i} \alpha _k 
	\\
	& \ge  \Bigl( 1 - \eta \Bigl( \frac{\varepsilon}{4N}\Bigr)\Bigr)  \sum _{ k \in C_i } \alpha _k  + \sum _{ k \in B\backslash C_i} \alpha _k  \sem \text{ (by \eqref{sum-alpha-D_i})} \\
	& \ge  \sum _{ k \in B } \alpha _k - \eta \Bigl( \frac{\varepsilon}{4N}\Bigr)    \\	
	& >  1- \sqrt{\eta ^\prime } - \frac{\varepsilon}{4N} \sem \text{(by \eqref{sum-alpha-B} and  condition c))}. 
	\end{align*}
	From the definition of $E$ and the previous chain of inequalities it follows that	
	\begin{equation}
	\label{sum-alpha-E}
	\sum _{ k \in E} \alpha _k 		 \ge 1 - N \sqrt{\eta ^\prime}
	- \frac{\varepsilon}{4} > 1 - \varepsilon.
	\end{equation}

	If  $i \le N$ and $C_i \ne \varnothing$ then  $D_i \ne \varnothing$  so  we can choose an element  $k_0^i  \in D_i$. In case that $C_i =\varnothing$  we choose $x_{i} \in S_{X_i}$ and $ x_{i}^* \in S_{ X_{i}^*}$ such that $x_{i}^* (x_i)=1$. For each  $k \in E$ we define $z_k \in S_{Z}$ as follows
	$$
	z_{k}(i) = \begin{cases}
	r_{k}(i) v_{k}(i)  \sep & \text{if} \sep k \in D_i \\
	r_{k}(i) v_{k_0^i}(i)  \sep & \text{if} \sep k \in  B \backslash C_i \sep \text{and} \sep C_i \ne \varnothing \\
	r_{k}(i) x_i  \sep &  \text{if}  \sep k \in  B \backslash C_i \sep \text{and} \sep C_i = \varnothing .
	\end{cases}
	$$ 
	Also we define $z^* \in Z^*$ by 
	$$
	z^*(i) = \begin{cases}
	s_i v_{i}^*  \sep & \text{if} \sep  C_i \ne \varnothing \\
	s_i x_{i}^*  \sep & \text{if} \sep  C_i = \varnothing .
	\end{cases}
	$$ 
	By Proposition \ref{pro-ab-norm-dual}, it is clear that  $z^* \in S_{Z^*}$ since $ 
	\Vert z ^* \Vert = \Vert (s_i) _{ i \le N}  \Vert _{ (\R^N) ^*} =   1.$

	In view of \eqref{vi*-vki-second} and \eqref{s-r-1-second} it is satisfied that
	\begin{equation}
	\label{z*-zk-2}
	z^* (z_k) = \sum  _{i=1}^N  z_{i}^* (z_{k}(i)) =    \sum  _{i=1}^N  s_i  r_{k}(i)=1, \sem \forall k \in E .
	\end{equation}


	Let us fix $k \in E $. If $k \in D_i$ we have
	\begin{eqnarray}
	\label{k-Di}
	\bigl\Vert z_k(i)  - u_k(i) \bigr\Vert  &=&  \bigl \Vert r_k(i) v_{k}(i)  -  u_k(i) \bigr \Vert \notag\\
	&=&   \Bigl \Vert r_k(i) v_{k}(i)  -  r_{k}(i)  \frac{u_k(i)}{ \Vert u_k(i) \Vert  } + r_{k}(i) \frac{u_k(i)}{ \Vert u_k(i) \Vert } - u_k(i) \Bigr \Vert   \notag \\
	&\leq&   \bigl \vert r_k(i)  \bigr \vert \Bigl \Vert v_{k}(i)  -    \frac{u_k(i)}{ \Vert u_k(i) \Vert  } \Bigr \Vert  +  \bigl \vert r_{k}(i)  - \bigl \Vert u_k(i)  \bigr \Vert \bigr \vert    
	\\
	&\le& \frac{\varepsilon}{4N}  +  \bigl \vert r_{k}(i)  - \bigl \Vert u_k(i)  \bigr \Vert \bigr \vert    \notag  \sem \text{(by \eqref{vi*-vki-second} and  condition c))}.
	\end{eqnarray}

	In case that  $k \in B \backslash C_i$ we obtain that
	\begin{eqnarray}
	\label{k-B-Ci}
	\bigl\Vert z_k(i)  - u_k(i) \bigr\Vert  &\le & 
	\bigl\Vert z_k(i)    \bigr\Vert + \bigl\Vert u_k(i) \bigr\Vert \notag \\
	&=&   \bigl\vert r_k(i)   \bigr\vert + \bigl\Vert   u_k(i) \bigr\Vert \notag \\
	&\leq&    \bigl\vert r_k(i)  - \bigl \Vert  u_k(i)  \bigr \Vert   \bigr\vert + 2\bigl\Vert   u_k(i)  \bigr\Vert 
	\\
	&\leq&    \bigl\vert r_k(i)  - \bigl \Vert  u_k(i)  \bigr \Vert   \bigr\vert + 2\sqrt[8]{\eta\prime}\notag \\
	&\leq&    \bigl\vert r_k(i)  - \bigl \Vert  u_k(i)  \bigr \Vert   \bigr\vert + \frac{\varepsilon}{4N} \notag.
	\end{eqnarray}
	By \eqref{k-Di} and \eqref{k-B-Ci}  we proved that for every $k \in E$ we have 
	\begin{equation}
	\label{k-E}
	\bigl\Vert   z_k(i) -u_k(i)  \bigr\Vert  \le   \bigl\vert r_k(i)  - \bigl \Vert  u_k(i)  \bigr \Vert   \bigr\vert + \frac{\varepsilon}{4N}.
	\end{equation}
	Taking into account  \eqref{s-r-1-second} for every $k \in E$  we deduce that 
	$$
	\Vert z_k - u_k \Vert \le \Bigl \vert \bigl( r_{k}(i) \bigr) _{i\le N} -   \bigl( \Vert  u_k(i)  \Vert  \bigr) _{i\le N}  \Bigr \vert + \sum _{i=1}^N \frac{\varepsilon}{4N} \le \frac{\varepsilon}{4N} + \frac{\varepsilon}{4} \le \frac{\varepsilon}{2} < \varepsilon.  
	$$
Since $z^* \in S_{ Z^*}$,  in  view of \eqref{sum-alpha-E}, \eqref{z*-zk-2} and the previous inequality the proof is also finished   in case 2. 
\end{proof}

 Let us notice  that the converse of Theorem \ref{th-stable} also holds.
	That is, in case that the product space $Z= \prod_{i=1}^N X_i$, endowed with an absolute    normalized norm, has the AHSp, then each space $X_i $  also has the AHSp for $1 \le i \le N$, a result proved in \cite[Theorem 2.3]{Ga}.

\section{ Stability of the approximate hyperplane property  under finite products}

The goal of this section is a result that asserts  the stability of a property stronger than the  approximate hyperplane series property under finite products endowed with an absolute norm.
 We   begin with the following notion that was introduced in
\cite[Definition 2.1]{CKLM}.

\begin{definition}
	\label{def-AHP}
	A Banach space $X$ has the  {\it approximate hyperplane property} (AHp) if there exists a
	function $\delta: ]0,1[ \longrightarrow  ]0,1[ $ and a $1$-norming subset $C$ of $S_{X^*}$
	satisfying the following property.
	
	Given   $\varepsilon> 0$ there is a function $\Upsilon_{X, \varepsilon}: C \llll
	S_{X^*}$ with  the following condition
	$$
	x^* \in C,\  x \in S_X, \ \Rea x^* (x) > 1 - \delta (\varepsilon) \ \Rightarrow \
	\dist(x, F(\Upsilon_{X, \varepsilon} (x^*))) < \varepsilon,
	$$
	where $F(y^*) = \{ y \in S_X : y^* (y)=1\}$ for any $y ^*\in S_{X^*}$.

	A family of Banach spaces $\{ X_i : i \in I\}$   has {\it AHp uniformly}  if every space $X_i$
	has property AHp with the same function $\delta$.
\end{definition}

Clearly we can assume that the $1$-norming subset $C$ in the previous definition satisfies
$\T C \subset C$, where $\T $  is the unit sphere of the scalar field.

Let us notice that a similar property to AHp was implicitly used to prove that several classes of
spaces have AHSp (see \cite{AAGM}).
It is known that property  AHp implies AHSp (see for instance \cite[Proposition 2.2]{CKLM}). It is an open question  whether or not the converse is true.
Examples of spaces having  AHp  are finite-dimensional spaces, uniformly convex spaces, $L_1
(\mu)$ for every measure $\mu$   and also $C(K)$ for every compact Hausdorff topological space
$K$ (see \cite[Propositions 3.5, 3.8, 3.6 and 3.7]{AAGM} and also \cite[Corollary 2.12]{CKLM}).

\begin{remark}
	\label{re-finite-AHP}	
	Let us notice that in view of Lemma \ref{le-AHSP-strong} the space  $\R^N,$ endowed with any norm,  satisfies AHp for the  $1$-norming set $S_{( \R^N )^*}$. Moreover if for some 
	$ 1 \leq i \leq N$  and $a^* \in S_{( \R^N )^*}$,  $a^*(e_i) = 0$, then $(\Upsilon_{\R^N ,\varepsilon} (a^*)) (e_i) = 0$, where $\Upsilon_{\R^N ,\varepsilon}$ is the mapping   appearing in Definition \ref{def-AHP}.
\end{remark}

The following result is a version of \cite[Lemma 2.9]{AMS}. 

\begin{lemma}
	\label{norming}
	Assume that $\vert \  \  \vert$ is an absolute normalized  norm on $\R^N$ and $X_i$  is a Banach space for $1 \le i \le N$.  If for each  $1 \le i \le N$,  $A_i \subset B_{X_i ^*}$  is a   $1$-norming set for $X_i$ such that $\T A_i \subset A_i$,  where $\T$ is the unit sphere of the scalar field,  then the set
	$$
	A= \Bigl\{ \bigl( r_i^*x_i^* \bigr) _{ i \le N} :   (r^*_i)_{ i \le N}  \in {S}_{ (\R^N)^*}, \;   r_ i^* \ge 0, \,  x_i^* \in A_i, \,  \forall 1 \le i \le N \Bigr\}
	$$
	is a $1$-norming set for $Z=  \prod_{i=1}^N  X_i,$  endowed with  the absolute norm  associated to $\vert \  \  \vert$.
\end{lemma}
\begin{proof}
	Assume that $(x_i) _{ i \le N} \in S_Z $ and $0 < \varepsilon  < 1$. By assumption for each  $1\leq i \leq N$ there is an element $x_i^* \in A_i $  satisfying that
	\begin{equation}
	\label{x-y-norm}
	\Rea x_i^* (x_i) \ge (1- \varepsilon ) \Vert x_i \Vert  \ge 0 .
	\end{equation}
	By Hahn-Banach theorem there is $(r_i^*) _{ i \le N} \in S_{ (\R ^N ) ^*}$  such that 
	\begin{equation}
	\label{r-x-1}
	\sum _{i=1}^N r_i^* \Vert x_i \Vert  =1.
	\end{equation}
	Clearly we can also assume that $r_i^* \in \R^+_{0}$ for each $1 \le i \le N$.  As a consequence we have that
	\begin{align*}
	\Rea \bigl(r_i^* x_i^*\bigr)_{i \le N}   \bigl(x_i  \bigr) _{i \le N} 	& =  \Rea \sum _{i=1}^N r_i^*   x_i^* (x_i) \\
	& \ge (1- \varepsilon )    \sum _{i=1}^N  r_i^*  \Vert x_i \Vert   \sem \text{(by \eqref{x-y-norm})} \\
	&   = 1 - \varepsilon  \sem \text{(by \eqref{r-x-1})} .
	\end{align*}
\end{proof}

\begin{theorem}
	\label{th-stable-AHP}
	Assume that  $\vert \ \;  \vert$ is an absolute  and  normalized norm on $\R^N$ and $ X_i$  is a Banach space satisfying  the approximate hyperplane property for each $1 \le i \le N$. Then the space
	$Z=  \prod_{i=1}^N  X_i,$  endowed with  the absolute norm  associated to $\vert \ \ \vert$,  also has the approximate hyperplane property.
\end{theorem}
\begin{proof}
	Without loss of generality we can assume that  for each $1 \le i \le N$, $X_i \ne \{0\}$ and $X_i$ has the AHp with a $1$-norming set $A_i \subset S_{X_i^*}$ such that  $\T A_i \subset A_i$. By   assumption and Remark \ref{re-finite-AHP}  there is a function $\eta : ]0,1[ \llll ]0,1[$ satisfying the following three conditions:
	\newline
	{\bf i)} the pair $(\varepsilon, \eta(\varepsilon) )$ satisfies the definition of AHp for the Banach space $X_i$ for each  $1 \leq i \leq N$.
	\newline
	{\bf ii)} the  space   $(\R^N, \vert \ \ \vert)$  satisfies the definition of AHp  with the function $\eta$   playing the role of $\delta $ and
	\newline
	{\bf iii)} $\eta (\varepsilon) < \varepsilon$ for each  $ \varepsilon \in ]0,1[$.

	Now we define $\eta ^\prime (\varepsilon)   = \eta  \biggl( \dfrac{ \eta^3 \bigl( \frac{\varepsilon}{4N} \bigr) }{4 N^2} \biggr) $. By Lemma \ref{norming} the set $A$ given by
	$$
	A= \Bigl\{ \bigl( r_i^*x_i^* \bigr) _{ i \le N} :   (r^*_i)_{ i \le N}  \in {S}_{ (\R^N)^*}, \;   r_ i^* \ge 0, \,  x_i^* \in A_i, \,  \forall 1 \le i \le N \Bigr\}
	$$
	is a $1$-norming set for $Z =  \prod _{ i=1}^N X_i$.  We will show that $Z $ has  the AHp with  the set $A$  and the function $\eta ^\prime $.
	
	We take an element  $\bigl(r_i^*x_i^*\bigr)_{ i \le N} \in  Z^* $ satisfying the conditions  in the definition of the set $A$  and $\bigl(x_i\bigr)_{i \le N} \in S_Z$
	such that 
	\begin{equation}
	\label{r-x*-x}
	\Rea \bigl(r_i^*x_i^*\bigr)_{ i \le N}  \bigl(x_i\bigr) _{ i \le N} > 1 - \eta ^\prime (\varepsilon).
	\end{equation}

	Define the set  $G $   by 
	$$
	G = \Bigl \{ i  \le N :  r_i^* >   \dfrac{ \eta \bigl( \frac{\varepsilon}{4N} \bigr)}{2N}  \Bigr\} \sem \text{and write} \sem G^c = \{ i \in \N : i \le N \} \backslash G.
	$$
	So 
	\begin{eqnarray}
	\label{r*-x*-x-eta}
	\sum _{i \in G} r_i^* \|x_i\|   &\geq &   \sum_{i \in G} r_i^* \Rea x_i^*(x_i) \notag\\
	&>&   1 - \eta^\prime (\varepsilon) - \sum_ {i \in G^c} r_i^* \Rea x_i^* (x_i)   \\
	&\geq &  1 - \eta ^\prime (\varepsilon)  - N \dfrac{ \eta \bigl( \frac{\varepsilon}{4N} \bigr)}{2N} \notag \\
	& >& 1 - \eta \Bigl( \frac{\varepsilon}{4N}\Bigr)>0 . \notag
	\end{eqnarray}

	Define $t^*= (t_i^*) _{ i \le N } =  \dfrac{P_G r^*}{ \vert P_G r^*\vert ^*} \in S_{(\R^N)^*}$, where we denoted by 
	$\vert \  \;  \vert ^*$ the  dual norm in $(\R ^N, \vert \  \;  \vert )^*$. In view of \eqref{r*-x*-x-eta}  we have that 
	$$
	\sum_{i \in G} t_i^* \|x_i\| > 1 - \eta \Bigl(\frac{\varepsilon}{4N}\Bigr ) > 0.
	$$
Now we use that  the space $(\R^N, \vert \ \  \vert)$ has the AHp (see condition ii))  and we write  $ \bigl( s_i^*\bigr) _{i \le N} = \Upsilon_{\R^N, \frac{\varepsilon}{4N} } (t^*)$.  In view of Remark   \ref{re-finite-AHP}  we know that  $s_i^* = 0$  if $i  \in G^c $  since  $t_{i}^* =0$  in this case.  
So  we obtain that 
	\begin{equation}
	\label{x-y-F}
	\dist 
	\Bigl( \bigl( \Vert x_i  \Vert  \bigr)_{ i \le N}, F  \bigl( \bigl( s_{i}^* \bigr)_{ i \le N }  \bigr)   \Bigr) 
	< \frac{\varepsilon }{4N}.
	\end{equation}
	So there is $(s_i) _{ i \le N}  \in  F  \bigl( \bigl( s_{i}^* \bigr)_{ i \le N }  \bigr)   \subset  S_{\R^N}$ satisfying 
	\begin{equation}
	\label{r1*s1-r1s1-xy}
	\vert (s_i) _{ i \le N } - (\Vert x_i \Vert )_{ i \le N} \vert < \frac{ \varepsilon}{4N}.
	\end{equation}
	
	Notice that 
	\begin{equation}
	\label{s*-s}
	\bigl( s_i^*\bigr)_{ i \le N }  \bigl( s_i\bigr)_{ i \le N }  =\sum_{i \in G} s_i^* s_i = 1.
	\end{equation}

	Now for each $1 \le i \le N$, we define $z_i \in S_{X_i} $ as follows.
	
	\textbf{Case 1.} Assume that  $i \in G $ and $\|x_i\| >  \dfrac{ \eta \bigl( \frac{\varepsilon}{4N} \bigr)}{2N}.$

	From \eqref{r-x*-x}  we obtain that
		\begin{align*}
	 1 - \eta ^\prime (\varepsilon) & <  \Rea \bigl(r_i^* x_i^*\bigr)_{i \le N}   \bigl(x_i  \bigr) _{i \le N} 	 \\
	 & =  \Rea \sum _{i=1}^N r_i^*   x_i^* (x_i) \\
	& \le     \sum _{i=1}^N  r_i^*    \Vert x_{i}^* \Vert \;\Vert x_i \Vert  \\
	&   \le    \sum _{i=1}^N  r_i^*   \Vert x_i \Vert  \\
	& \le 1.
	\end{align*}
	As a consequence we get that 	
	$$
	\Rea r_i^* x_i^* (x_i) > r_i^* \Vert x_i \Vert - \eta ^\prime (\varepsilon),  
	$$
	and so 
	$$
	\Rea x_i^* \Bigl( \frac{x_i}{\Vert x_i \Vert }\Bigr) > 1 - \frac{ \eta ^\prime (\varepsilon) }{ r_i^* \Vert x_i \Vert } > 1 -  \eta \Bigl( \frac{\varepsilon}{4N} \Bigr) .
	$$
	
	Since we  assume that $X_i$ has the AHp with the function $\eta$ and the subset $A_i$ , we conclude that
	$$
	\dist \Bigl( \frac{x_i} {\Vert x_i \Vert} , F \bigl( \Upsilon _{X_i, \frac{\varepsilon}{4N}} (x_i^* ) \bigr)   \Bigr ) < \frac{\varepsilon}{4N} . 
	$$
	So there is $z_i \in   F \bigl( \Upsilon _{X_i, \frac{\varepsilon}{4N}} (x_i^* ) \bigr) $ such that 
	$ \Bigl \Vert z_i - \dfrac{x_i}{\|x_i\|} \Bigr \Vert  < \dfrac{\varepsilon}{4N}$.  As a consequence we have that  
	$ \bigl \Vert\,  \Vert x_i \Vert z_i - x_i \bigr \Vert  < \dfrac{\varepsilon}{4N}$. In view of  \eqref{r1*s1-r1s1-xy}  we deduce that  have 
	\begin{equation}
	\label{si-zi-xi-1}
	\|s_i z_i - x_i\| < \dfrac{\varepsilon}{2N} .
	\end{equation}
	
	\textbf{Case 2.} Assume that  $i \in G$ and $\|x_i\| \leq \dfrac{\eta (\dfrac{\varepsilon}{4N})}{2N}.$

	We choose an element  $z_i \in F  \bigl( \Upsilon _{X_i,\frac{\varepsilon}{4N}} (x_i^*) \bigr) $.  From \eqref{r1*s1-r1s1-xy} we have 
	\begin{equation}
	\label{si-zi-xi-2}
	\|s_i z_i - x_i\| \leq |s_i| + \|x_i\| < \frac{2 \varepsilon}{4N} + \frac{\varepsilon}{4N} = \frac{3\varepsilon}{4N}.
	\end{equation}
	
	\textbf{Case 3.} Assume that $i \in G^c$  and $ x_i \neq 0 .$
	
	Define $z_i = \dfrac{x_i}{\|x_i\|}.$  By  \eqref{r1*s1-r1s1-xy} we have 
	\begin{equation}
	\label{si-zi-xi-3}
	\|s_i z_i - x_i \| = \Bigl \Vert s_i \dfrac{x_i}{\|x_i\|} - x_i \Bigr \Vert  = |s_i - \|x_i\|| < \dfrac{\varepsilon}{4N}.
	\end{equation}

	\textbf{Case 4.} Assume that $i \in G^ c$  and $ x_i = 0 .$
	
	In this case we choose any element $z_i \in S_{X_i}$. In view of  \eqref{r1*s1-r1s1-xy} we have 
	\begin{equation}
	\label{si-zi-xi-4}
	\|s_i z_i - x_i\| = |s_i| < \dfrac{\varepsilon}{4N}.
	\end{equation}

	So  from \eqref{si-zi-xi-1}, \eqref{si-zi-xi-2}, \eqref{si-zi-xi-3}  and \eqref{si-zi-xi-4} we conclude  that 
	$$
	\Vert \bigl( s_i z_i \bigr) _{ i \le N }  -  \bigl( x_i \bigr) _{ i \le N }  \Vert  \le \sum _{i=1}^N \Vert s_i z_i - x_i \Vert <   \frac{3N\varepsilon}{4N} <  \varepsilon.
	$$
	Notice that 
	$$
	\bigl( s_i^*  \Upsilon _{X_i,\frac{\varepsilon}{4N}} (x_i^* ) \bigr) _{ i \le N}  \in S_{Z^*}  \sem \text{and} \sem (s_i z_i) _{ i \le N } \in S_Z .
	$$
	From \eqref{s*-s} we have 
	$$ 
	\bigl( s_i^*  \Upsilon _{X_i,\frac{\varepsilon}{4N}} (x_i^* ) \bigr) _{ i \le N}   \bigl( s_i z_i  \bigr) _{ i \le N}    =\sum _{i \in G} s_i^* s_i = 1.
	$$
	So   the proof is finished. 
\end{proof}

In \cite[Theorem 2.10]{AMS} the authors provided a stability result of AHSp  under  some infinite sums that includes $\ell_p$-sums for $1 \le p< \infty$. Here we provide  a simple  example showing  that in   such stability result  some
	requirement on the  Banach lattice sequence   used to define the infinite sum of Banach spaces   is needed.   For that example we need the following easy result.

\begin{lemma} 
	\label{le-ineq-ell1}	
	It is satisfied that  $\Bigl \Vert \Bigl( \frac{ x_n }{2^n} \Bigr) \Bigr \Vert _2 \leq  \Vert x  \Vert_1  $ for any  element $x \in \ell_1$.	
\end{lemma}
\begin{proof}
	If $x \in \ell_1$,  it is clear that  
	\begin{align*}
	\Bigl \Vert \Bigl( \frac{ x_n }{2^n} \Bigr) \Bigr \Vert _2	& \le  \Vert x  \Vert_\infty  	\Bigl \Vert \Bigl( \frac{1}{2^n} \Bigr) \Bigr \Vert  _2   \\
	&   \le  \Vert x  \Vert_1 	\Bigl \Vert \Bigl( \frac{1}{2^n} \Bigr) \Bigr \Vert  _1	\\
	&  =  \Vert x  \Vert_1 .
	\end{align*}
\end{proof}

We need to recall  some notions.    In order to do this we  denote by $\omega$ the space of all  real sequences.    A real Banach space $E\subset \omega$ is  {\it solid} whenever $x \in w$, $y \in E$ and $|x|
\leq |y|$ then $x \in E$ and $\|x\|_E \leq \|y\|_E$. $E$ is said to be a~\emph{Banach sequence
	lattice}   if  $E\subset \omega$, $E$ is solid and there exists $u \in
E$ with $u>0$.

Let $E$ be a~Banach sequence lattice. For a~given sequence $(X_k,
\|\ \|_{X_k})_{k=1}^{\infty}$ of Banach spaces the linear space  of sequences $x =
(x_k)$, with $x_k\in X_k$ for each $k\in \mathbb{N}$ and satisfying that  $(\|x_k\|_{ X_k})\in E$,
becomes a~Banach space endowed with the norm
\[
\|(x_k)\| = \big\|\big(\|x_k\|_{X_k}\big)\big\|_{E}.
\]
We  denote the previous   space  by $\big(\oplus \sum_{k=1}^{\infty} X_k\big)_E$.
Finally we recall that a Banach lattice  $E$ is {\it uniformly monotone}  if for each 
$\varepsilon > 0$ there is $\delta> 0$ satisfying  the following condition
$$
x \in S_E,  y \in E, \  x,y \ge 0, \ \Vert x +y \Vert \le 1 + \delta \ \Rightarrow \ \Vert y \Vert \le \varepsilon .
$$

\begin{example}
	\label{ell1-sum-without-AHSP}
	The space $E= \ell_1$, endowed with the norm $\Vert \ \  \Vert$ given by
	$$
	\Vert x \Vert =  \Vert x \Vert_1 +  \Bigl \Vert \Bigl( \frac{ x_n }{2^n} \Bigr) \Bigr \Vert _2 
	\ \ \ \  (  x \in E)
	$$	
	is a  uniformly monotone Banach lattice sequence without the AHSp and so it does not satisfy the AHp.
\end{example}
\begin{proof}
	One can easily check that  $E$ is a Banach lattice sequence and $\Vert \ \ \Vert$ is a strictly convex norm equivalent to the usual norm of $\ell_1$. 
	
	Since the norm $\Vert \ \  \Vert$ is equivalent to the usual norm of $\ell_1$, $E$ is not  reflexive and  so the norm  $\Vert \ \  \Vert$ is not uniformly convex. By 
	\cite[Proposition 3.9]{AAGM} the space $E$ does not have the AHSp and so  it cannot satisfy the AHp by 
	\cite[Proposition 2.2]{CKLM}.

	Now we show that the  Banach lattice sequence  space  $E$ is uniformly monotone. 	
	Assume $\varepsilon > 0$, $x \in S_E$ and  $y \in E$ such that $x, y \geq 0$ and $\Vert x+y \Vert \leq 1 + \dfrac{\varepsilon}{2}.$  We  will show  that $\Vert y \Vert \leq \varepsilon $, so $E$ is uniformly monotone. 
	
	Since  $x, y \geq 0,$  notice that 	$\Vert x + y  \Vert _1 = \Vert x  \Vert _1 + \Vert y \Vert _1$ and $\Bigl \Vert \Bigl( \frac{ x_n }{2^n} \Bigr) \Bigr \Vert _2 \leq  \Bigl \Vert \Bigl( \frac{ x_n + y_n }{2^n} \Bigr) \Bigr \Vert _2$, therefore 
	\begin{align*}
	1 + \Vert  y  \Vert _1 & =  \Vert x  \Vert +  \Vert  y  \Vert _1     \\
	& =  \Vert x \Vert _1 + \Bigl \Vert \Bigl( \frac{ x_n }{2^n} \Bigr) \Bigr \Vert _2 + \Vert  y  \Vert _1   \\
	&  \leq  \Vert x \Vert _1 + \Bigl \Vert \Bigl( \frac{ x_n + y_n }{2^n} \Bigr) \Bigr \Vert _2 + \Vert  y  \Vert _1   \\
	& = \Vert x + y  \Vert _1 + \Bigl \Vert \Bigl( \frac{ x_n + y_n }{2^n} \Bigr) \Bigr \Vert _2   \\
	& =  \Vert x + y \Vert   \\
	&  \leq  1 + \dfrac{\varepsilon}{2} 
	\end{align*}	
	So $\Vert y \Vert _1 \leq  \dfrac{\varepsilon}{2}  $ and by  Lemma \ref{le-ineq-ell1}  we have that $\Vert y \Vert \le  \varepsilon. $  
\end{proof}

\begin{corollary}
	There  exist a uniformly monotone Banach sequence lattice  $E$  and a family of Banach spaces  $\{X_k :  k \in \mathbb{N}\}$ satisfying AHp uniformly such that  $(\oplus \sum_{k=1}^{ \infty} X_k)_E$ does not have AHSp, so it does not have AHp.
\end{corollary}

\begin{proof}
	Assume  that $E$  is  the uniformly monotone Banach sequence lattice introduced in the previous example, and we take  $X_k = \mathbb{R} $ for every positive integer $k$. So   the family $\{X_k : k \in \mathbb{N}\}$ satisfies  AHp uniformly and  	 $(\oplus \sum_{k=1}^{ \infty} X_k)_E = E$. Hence  $(\oplus \sum_{k=1}^{ \infty} X_k)_E $ does not have AHSp.  Since AHp implies AHSp by \cite[Proposition 2.2]{CKLM},   $(\oplus \sum_{k=1}^{ \infty} X_k)_E $ does not satisfy  AHp.
\end{proof}

\bibliographystyle{amsalpha}

} 

\end{document}